\documentclass[preprint,12pt,authoryear]{elsarticle}

\usepackage{amsmath,amsfonts,amssymb,amsthm,mathrsfs,bbm}
\usepackage[colorlinks=true, linkcolor=blue, citecolor=blue, urlcolor=blue]{hyperref}
\usepackage{paralist}
\usepackage{enumitem}
\usepackage{textcase} 
\usepackage{titlesec} 
\usepackage{etoolbox} 
\usepackage{float}
\usepackage{graphics}
\usepackage{booktabs}
\usepackage{multirow}
\usepackage{xcolor}

\numberwithin{equation}{section}

\theoremstyle{plain}
\newtheorem{theorem}{Theorem}[section]

\newtheorem{lemma}[theorem]{Lemma}

\theoremstyle{definition}
\newtheorem{remark}[theorem]{Remark}

\newtheorem{assumption}[theorem]{Assumption}

\begin{document}

\begin{frontmatter}



\title{A note on ``The volume of random simplices from elliptical distributions in high dimension''} 

\author{Xizheng Shan}
\address{School of Economics and Management, Tsinghua University, 100084, China}

\author{Yanpeng Li} 
\ead{20230256@hit.edu.cn}
\address{{School of Mathematics, Harbin Institute of Technology},{Harbin},{150001}, {China}}

\begin{abstract}
Recent work by Gusakova et al. (Stochastic Process. Appl. 164 (2023) 357–382) has shown a central and a stable limit theorem for the logarithmic volume of random simplices and random convex bodies under an elliptical framework in the high dimensional regime, that is, if $p\rightarrow\infty$ and $n\rightarrow \infty$ in such a way that $p/n\rightarrow \gamma\in (0,1)$. A technical condition (Equation (2.6) of Assumption (B) therein) requires that the population matrix $\mathbf{A}\mathbf{A}^{\top}$ is close in Frobenius norm to a multiple of the identity matrix, which is rather restrictive and rules out various settings for statistical application, such as spiked models and dependent structure models. In this note we offer a general relaxation of this condition, which arrives at a reasonable condition and covers numerous scenarios, as well as consequences for the volume of general random simplices and random convex bodies. In particular, our results covers the Toeplitz/AR(1) covariance structures studied by Jiang and Pham (Ann. Stat. 53 (2025) 907-928), giving a concrete application of our theorem to high-dimensional dependent covariance models.
\end{abstract}




\end{frontmatter}




\section{Introduction}\label{sec_introduction}
The probabilistic analysis of convex hulls in $\mathbb{R}^n$ generated by $p$ random points constitutes a central theme in geometric probability and stochastic geometry. Principal lines of inquiry include the asymptotic behavior of expectations for key parameters, such as the number of faces or the intrinsic volumes, and their profound connections to affine surface areas. Related research further encompasses sharp upper and lower bounds for variances, as well as results concerning the asymptotic normality and concentration properties of these combinatorial and geometric quantities; see \cite{adamczak2024limit,barany2007central,gusakova2023volume} and references therein. Of particular interest is the model where the points are drawn from a class of $q$-radial distributions on the $\ell_q$-ball in $\mathbb{R}^n$, often referred to as the study of high-dimensional pinned random simplices anchored at the origin.

Recent work by \cite{gusakova2023volume} has shown a central limit theorem for the logarithmic volume of $p$-dimensional pinned random simplices whose generating points follow a general elliptical
distribution in $\mathbb{R}^n$ under a high-dimensional regime. 
Consider an $n$-dimensional elliptical population 
\begin{equation}\label{eq_defx}
    \mathbf{x}=\xi \mathbf{A}\mathbf{u}\in \mathbb{R}^{n},
\end{equation} 
where $\xi\ge 0$ is a scalar random variable representing the radius of $\mathbf{x}$, $\mathbf{A}\in\mathbb{R}^{n\times n}$ is a deterministic matrix of full rank, and $\mathbf{u}\in \mathbb{R}^n$ is the random direction, which is uniformly distributed on the unit sphere $\mathbb{S}^{n-1}$. Moreover, we suppose that $\xi$ and $\mathbf{u}$ are independent. Given an independent and identically distributed (i.i.d.) sample $\mathbf{x}_1,\ldots,\mathbf{x}_p$ with $\mathbf{x}_i=\xi_i\mathbf{A}\mathbf{u}_i$ for $1\le i\le p$, the data matrix $\mathbf{X}=(\mathbf{x}_1,\ldots,\mathbf{x}_p)^{\top}$ collects the elliptical random vectors. Then the (pinned) random simplex $\Delta \mathbf{X}$ with vertex $\{\mathbf{0},\mathbf{x}_1,\ldots,\mathbf{x}_p\}$ is defined as
\begin{equation}\label{eq_defsimplex}
    \Delta \mathbf{X}=:\{\sum_{i=1}^ps_i\mathbf{x}_i:s_i\ge 0 ~\text{and}~\sum_{i=1}^ps_i=1\},
\end{equation}
and its $p$-volume admits the following representation in terms of the data matrix $\mathbf{X}$:
\begin{equation}\label{eq_defvolume}
    \operatorname{Vol}_p(\Delta \mathbf{X})=\frac{1}{p!}\sqrt{\det (\mathbf{X}\mathbf{X}^{\top})},
\end{equation}
which further implies 
\begin{equation}\label{eq_logdefvolume}
    \log\operatorname{Vol}_p(\Delta \mathbf{X})=-\log (p!)+\frac{1}{2}\log\det (\mathbf{X}\mathbf{X}^{\top}).
\end{equation}
In the case of $p\le n$, the random vectors $\mathbf{x}_1,\ldots,\mathbf{x}_p$ are almost surely linearly independent, and thus $\Delta \mathbf{X}$ is a $p$-dimensional random convex polytope in $\mathbb{R}^n$ with non-zero $p$-volume. Defining the diagonal matrix $\mathbf{D}$ with diagonal entries $\xi_1, \ldots, \xi_p$, and the matrix
\begin{align*}
\mathbf{Y}:=\mathbf{Y}_n:=\left(\mathbf{A} \mathbf{u}_1, \ldots, \mathbf{A} \mathbf{u}_p\right)^{\top},
\end{align*}
we have
\begin{align*}
\operatorname{det}\left(\mathbf{X} \mathbf{X}^{\top}\right)=\operatorname{det}\left(\mathbf{D Y} \mathbf{Y}^{\top} \mathbf{D}\right)=\operatorname{det}\left(\mathbf{Y} \mathbf{Y}^{\top}\right)(\operatorname{det} \mathbf{D})^2
\end{align*}
and taking the logarithm yields
\begin{align*}
\log \operatorname{det}\left(\mathbf{X} \mathbf{X}^{\top}\right) =2 \log \operatorname{det} \mathbf{D}+\log \operatorname{det}\left(\mathbf{Y} \mathbf{Y}^{\top}\right) 
=2 \sum_{i=1}^p \log \xi_i+\log \operatorname{det}\left(\mathbf{Y} \mathbf{Y}^{\top}\right),
\end{align*}
where the last two terms are independent and can be analyzed separately. For the first part, due to the i.i.d. assumption, the generalized central limit theorem for sums of i.i.d. random variables implies that, as $n\rightarrow\infty$ there are sequences $(m_n)_{n\ge 1}$ and $(s_n)_{n\ge 1}$ such that 
\begin{equation}\label{eq_cltxi}
    \frac{\sum_{i=1}^{p}\log \xi_i-m_n}{s_n}\stackrel{d}{\rightarrow} \mu_{\alpha},
\end{equation}
where $\mu_{\alpha}$ stands for the $\alpha$-stable distribution with stability index $\alpha\in (0,2]$. For instance, the Cauchy distribution $\alpha=1$ and Gaussian distribution $\alpha=2$. 

For the log determinant $\log\det(\mathbf{Y} \mathbf{Y}^{\top})$, given $p\rightarrow\infty, n\rightarrow\infty$ and $p/n\rightarrow \gamma\in (0,1)$, \cite{gusakova2023volume} established a CLT for $\log \det(\mathbf{Y}\mathbf{Y}^{\top})$ with a technical assumption for the population matrix $\mathbf{A}\mathbf{A}^{\top}$ as follows:\\
\textbf{Assumption (B)} There exists a constant $C \ge 1$ not depending on $n$ such that the ordered eigenvalues
\begin{equation}\label{eq_lambdabounded}
\lambda_{\max }(\mathbf{A} \mathbf{A}^{\top})=\lambda_1(\mathbf{A} \mathbf{A}^{\top}) \geq \lambda_2(\mathbf{A} \mathbf{A}^{\top}) \geq \cdots \geq \lambda_n(\mathbf{A} \mathbf{A}^{\top})=\lambda_{\min }(\mathbf{A} \mathbf{A}^{\top})
\end{equation}
of $\mathbf{A} \mathbf{A}^{\top}=\mathbf{A}_n \mathbf{A}_n^{\top}$ satisfy $C^{-1} \le \lambda_{\min }(\mathbf{A} \mathbf{A}^{\top}) \le \lambda_{\max }(\mathbf{A} \mathbf{A}^{\top}) \le C$. In addition, we assume that
\begin{equation}\label{eq_lambdaequal}
\lim _{n \rightarrow \infty} \operatorname{tr}\left(\mathbf{A}_n \mathbf{A}_n^{\top}-\frac{\operatorname{tr}\left(\mathbf{A}_n \mathbf{A}_n^{\top}\right)}{n} \mathbf{I}_n\right)^2=\lim _{n \rightarrow \infty} \sum_{k=1}^n\left(\lambda_k\left(\mathbf{A}_n \mathbf{A}_n^{\top}\right)-\frac{1}{n} \sum_{j=1}^n \lambda_j\left(\mathbf{A}_n \mathbf{A}_n^{\top}\right)\right)^2=0,
\end{equation}
where $\mathbf{I}_n$ is the $n \times n$ identity matrix. This condition \eqref{eq_lambdaequal} is restrictive, as it means that the population matrix $\mathbf{A}\mathbf{A}^{\top}$ is nearly a multiple of the identity matrix and thus rules out various applications such as the likelihood ratio test in statistics (see \cite{hu2019high,onatski2013asymptotic} and references therein). Let us begin with a toy example in the spiked model \cite{johnstone2001distribution}. For the simplest spiked model, the covariance matrix can be expressed as $\Sigma=\sigma^2\left(\mathbf{I}_n+h \mathbf{v} \mathbf{v}^{\top}\right)$ where $\sigma^2$ is an unknown parameter, $h$ is a constant, and $v$ is a unit vector in $\mathbb{R}^n$. It is obvious that all the cases $h>0$ are excluded from \eqref{eq_lambdaequal}. Moreover, as pointed out in Remark 2.2 of \cite{gusakova2023volume}, \eqref{eq_lambdaequal} is only required to control the variance of a linear combination of diagonal entries of a large projection matrix and can be dropped if (A.5) (the variance of the diagonal entries of a large projection matrix) can be improved to $o(1/n)$ in the proof of Lemma A.1 of \cite{gusakova2023volume}. We show that the upper bound for (A.5) is optimal for some large projection matrices even for the trivial case that $\mathbf{A}=\mathbf{I}_n$, see Remark \ref{remark_lowerbound} for details.

Although our methods are similar to pre-existing results, we relax the  
restrictive condition \eqref{eq_lambdaequal} to a general setting as it contains the block spiked models \cite{johnstone2001distribution}, Toeplitz models \cite{jiang2025asymptoticaos}, interclass correlation models \cite{heiny2026maximum} and so on. As a consequence, one can derive the central limit theorem for the volume of random simplices and convex bodies for general settings following \eqref{eq_cltxi} and the techniques by \cite{gusakova2023volume}. The rest of the paper is organized as follows. In Section \ref{sec_main}, we present the main results followed by an application for the likelihood ratio test of the spiked model. Section \ref{sec_proof} consists of the detailed proofs of our theoretical results. Section \ref{proof_pikk} provides the proof for the auxiliary lemma \ref{lemma_pikk}, which is a key input to study the logarithmic determinant.

\section{Main result}\label{sec_main}
Recalling the model in \eqref{eq_defx}, we further make the following two assumptions.
\begin{assumption}\label{assump0}
    Assume $p,n\rightarrow \infty$ with $p/n\rightarrow \gamma\in (0,1)$.
\end{assumption}
\begin{assumption}\label{assump1}
Assume that the eigenvalues of $\mathbf{A}^{\top}\mathbf{A}$ satisfy $\tau\le \lambda_{\min}(\mathbf{A}^{\top}\mathbf{A})\le \lambda_{\max}(\mathbf{A}^{\top}\mathbf{A})\le \tau^{-1}$ for some constant $\tau\in (0,1)$. Moreover, suppose that the ordered eigenvalues $\lambda_1(\mathbf{A}\mathbf{A}^{\top})\ge \lambda_2(\mathbf{A}\mathbf{A}^{\top})\ge\cdots\ge \lambda_n(\mathbf{A}\mathbf{A}^{\top})$ satisfy 
\begin{equation}\label{eq_condition}
    \sum_{k=1}^{n}(\lambda_k(\mathbf{A}\mathbf{A}^{\top})-\bar{\lambda})^2=o(n)~\text{and}~\sum_{k\ne \ell}^{n}|(\lambda_k(\mathbf{A}\mathbf{A}^{\top})-\bar{\lambda})(\lambda_{\ell}(\mathbf{A}\mathbf{A}^{\top})-\bar{\lambda})|=o(n),
\end{equation}
where $\bar{\lambda}=:\operatorname{tr}(\mathbf{A}\mathbf{A}^{\top})/n=n^{-1}\sum_{k=1}^{n}\lambda_k(\mathbf{A}\mathbf{A}^{\top})$.
\end{assumption}

Now we state the CLT for the logarithmic determinant of $\mathbf{Y}\mathbf{Y}^{\top}$, which is the main result of this paper.
\begin{theorem}\label{thm_logdet}
	Let $(\mathbf{Y})_{n \geq 1}$ be a sequence of random $p \times n$ matrices defined as follows: $\mathbf{Y}=\left(\mathbf{A u}_1, \ldots, \mathbf{A u}_p\right)^{\top}$, where $(\mathbf{A})_{n \geq 1}$ is a sequence of deterministic $n \times n$ matrices satisfying Assumption \ref{assump1}, and $\mathbf{u}_1, \ldots, \mathbf{u}_p$ are independent $n$-dimensional random vectors, distributed uniformly on the unit sphere $\mathbb{S}^{n-1}$. Under Assumption \ref{assump0}, as $n \rightarrow \infty$, it holds that
\begin{equation*}
\frac{\log \operatorname{det}(\mathbf{Y} \mathbf{Y}^{\top})-\mu_n}{\sigma_n} \xrightarrow{d} N(0,1),
\end{equation*}
where the centering and normalizing sequences $(\mu_n)_{n \ge 1}$ and $(\sigma_n)_{n \ge 1}$ are given by
\begin{align*}
\begin{aligned}
\mu_n & =\log \operatorname{tr}(\mathbf{A} \mathbf{A}^{\top})-p \log n-\frac{\sigma_n^2}{2}+\sum_{i=1}^{p-1} \log \left(\sum_{k=1}^n \lambda_k\left(\mathbf{A} \mathbf{A}^{\top}\right) t_{i, k}(\mathbf{A})\right), \\
\sigma_n^2 & =-2 \frac{p}{n}+2 \sum_{i=1}^{p-1} \frac{\sum_{k=1}^n \lambda_k^2\left(\mathbf{A} \mathbf{A}^{\top}\right) t_{i, k}(\mathbf{A})}{\left(\sum_{k=1}^n \lambda_k\left(\mathbf{A} \mathbf{A}^{\top}\right) t_{i, k}(\mathbf{A})\right)^2},
\end{aligned}
\end{align*}
where \begin{align*}
t_{i, k}(\mathbf{A})=\mathbb{E}\left[\frac{1}{1+\lambda_k\left(\mathbf{A} \mathbf{A}^{\top}\right) \mathbf{w}_{i k}^{\top}\big(\sum_{\ell=1 ; \ell \neq k}^n \lambda_{\ell}\left(\mathbf{A} \mathbf{A}^{\top}\right) \mathbf{w}_{i \ell} \mathbf{w}_{i \ell}^{\top}\big)^{-1} \mathbf{w}_{i k}}\right],
\end{align*}
with $\mathbf{w}_{i 1}, \ldots, \mathbf{w}_{i n}$ being i.i.d. $i$-dimensional random column vectors whose components are independent standard normal random variables. 
\end{theorem}
\begin{remark}
Compared with \cite{gusakova2023volume}, our condition \eqref{eq_condition} is softer and contains the spiked models in general. Specifically, the number of spiked eigenvalues can be divergent, for instance, $\mathbf{A} \mathbf{A}^{\top}=:\mathbf{D}+\mathbf{I}_n$ where $\|\mathbf{D}\|\le C$ and $\mathrm{rank}(\mathbf{D})\le o(n^{1/2})$ for some constant $C>0$. For general dense cases, our setting also contains the Toeplitz matrix model. Specifically, as illustrated by \cite{jiang2025asymptoticaos}, consider the covariance matrix of the auto-regressive model $AR(1)$,  
\[
\mathbf{A} \mathbf{A}^{\top}=(r^{|i-j|})_{1\le i,j\le n},
\]
where $0\le r<1$ can be $n$-dependent. One can check that \eqref{eq_condition} can be satisfied if $r=o(n^{-1/2})$ since one has $\frac{1-r^2}{(1+r)^2}\le \lambda_{\min }\le \lambda_{\max }\le \frac{1+r}{1-r}$ by Lemma A.1 of \cite{jiang2025asymptoticaossupp}. Moreover, consider the (weak) interclass correlation model \cite{heiny2026maximum}
\[
\mathbf{A} \mathbf{A}^{\top}=(1-\rho)\mathbf{I}_n+\rho\mathbf{1}\mathbf{1}^{\top},
\]
where $0\le \rho<1$ can be $n$-dependent. It can be verified that \eqref{eq_condition} is also satisfied if $\rho=C/n$ for some constant $C$, which is excluded by \eqref{eq_lambdaequal}. Thus, the interclass/equicorrelation structure studied in \cite{heiny2026maximum} provides another substantive dependent-model application of our relaxed condition: its eigenvalue structure verifies Assumption \ref{assump1} while falling outside the near-isotropic condition \eqref{eq_lambdaequal} used in \cite{gusakova2023volume}. Consequently, Theorem \ref{thm_logdet} extends the log-determinant and random-simplex volume CLT to the equicorrelated high-dimensional dependence regime analyzed in \cite{heiny2026maximum} for Gaussian extremes, sample coefficients, and multiple-testing applications.
\end{remark}

\section{Proof of Theorem \ref{thm_logdet}}\label{sec_proof}
The proof of Theorem \ref{thm_logdet} follows the arguments of \cite{gusakova2023volume}, which relies on the method of perpendiculars \cite{girko1998refinement} and the analysis of the diagonals of a sequence of large projection matrices based on the moments of Wishart matrices by \cite{holgersson2020recent,pielaszkiewicz2020mixtures}. To this end, we only show that Lemma A.1 of \cite{gusakova2023volume} (see Lemma \ref{lemma_traceQ} below) holds true under Assumptions \ref{assump0} and \ref{assump1}, where other proofs are the same as that of \cite{gusakova2023volume} and thus omitted. 

\subsection{Preminiliary}\label{sec_main2}
Before we present the main result, we collect some notation that is useful in the sequel. For any symmetric matrix $\mathbf{M} \in \mathbb{R}^{d \times d}$, we will write
\begin{align*}
\lambda_{\max }(\mathbf{M})=\lambda_1(\mathbf{M}) \ge \lambda_2(\mathbf{M}) \ge \cdots \ge \lambda_d(\mathbf{M})=\lambda_{\min }(\mathbf{M})
\end{align*}
for its ordered eigenvalues, $\|\mathbf{M}\|$ for its spectral norm, and denote its spectral decomposition by
\begin{align*}
\mathbf{M}=\mathbf{O}_{\mathbf{M}} {\Lambda}_{\mathbf{M}} \mathbf{O}_{\mathbf{M}}^{\top},
\end{align*}
where ${\Lambda}_{\mathbf{M}}$ is the diagonal matrix whose $i$-th diagonal element is $\lambda_i(\mathbf{M})$, and $\mathbf{O}_{\mathbf{M}}$ is an orthogonal matrix. Given two sequences $\left(a_n\right)_{n \in \mathbb{N}}$ and $\left(b_n\right)_{n \in \mathbb{N}}$ we write $a_n=O\left(b_n\right)$ if $\limsup _{n \rightarrow \infty}\left|a_n / b_n\right|<\infty$, $a_n=o\left(b_n\right)$ if $\lim _{n \rightarrow \infty}\left|a_n / b_n\right|=0$, and $a_n\asymp b_n$ if there exists some constant $c>0$ such that $c<|a_n/b_n|<c^{-1}$. Denote $a_n\lesssim b_n$ is there is some constant $C>0$ such that $a_n\le C b_n$ as $n\rightarrow\infty$.

Next, due to the rotational invariance of the rows of $\mathbf{U}=(\mathbf{u}_1,\ldots,\mathbf{u}_p)^{\top}$, we can represent the sample covariance matrix $\mathbf{Y}\mathbf{Y}^{\top}$ as 
\begin{equation*}
\mathbf{Y}\mathbf{Y}^{\top}=\mathbf{U}\mathbf{A}^{\top}\mathbf{A}\mathbf{U}^{\top}\stackrel{d}{=}\mathbf{U}\Lambda \mathbf{U}^{\top},
\end{equation*}
where the diagonal matrix $\Lambda=\operatorname{diag}(\lambda_1(\mathbf{A}^{\top}\mathbf{A}),\ldots,\lambda_n(\mathbf{A}^{\top}\mathbf{A}))$. Therefore, we may assume without loss of generality that $\mathbf{A}$ is a
diagonal matrix with $A_{kk}=\sqrt{\lambda_k(\mathbf{A}^{\top}\mathbf{A})}\in (\tau,\tau^{-1})$ for some constant $\tau\in (0,1)$. Noticing the stochastic representation for the random direction $\mathbf{u}$, we can represent 
\begin{equation}\label{eq_defUN}
\mathbf{U} \stackrel{d}{=}\left(\operatorname{diag}(\mathbf{N} \mathbf{N}^{\top})\right)^{-1 / 2} \mathbf{N},
\end{equation}
where $\mathbf{N}$ is a $p \times n$ matrix with independent standard normal entries ($N_{i j}$) and $\operatorname{diag}(\mathbf{N N}^{\top})$ denotes the diagonal matrix with the same diagonal elements as $\mathbf{N N}^{\top}$. 
Thus, we have
\begin{align*}
\log \operatorname{det}\left(\mathbf{U} \mathbf{A}^{\top} \mathbf{A} \mathbf{U}^{\top}\right)=p \log \frac{\operatorname{tr}\left(\mathbf{A}^{\top} \mathbf{A}\right)}{n}+\log \operatorname{det}\left(\frac{n}{\operatorname{tr}\left(\mathbf{A}^{\top} \mathbf{A}\right)} \mathbf{U A}^{\top} \mathbf{A} \mathbf{U}^{\top}\right) .
\end{align*}
Thus, in the remainder of this section, we will assume that  $\operatorname{tr}\left(\mathbf{A}^{\top} \mathbf{A}\right)=n$. 

Using Girko's method of perpendiculars \cite{girko1998refinement} gives
\begin{align*}
\log \operatorname{det}\left(\mathbf{Y} \mathbf{Y}^{\top}\right)=-p \log n+\sum_{i=0}^{p-1} \log \left(Z_{i+1}\right),
\end{align*}
where
\begin{align*}
Z_{i+1}=n \mathbf{b}_{i+1}^{\top} \mathbf{P}_i \mathbf{b}_{i+1} \quad \text { and } \quad \mathbf{P}_i=\mathbf{I}_n-\mathbf{B}_{(i)}^{\top}(\mathbf{B}_{(i)} \mathbf{B}_{(i)}^{\top})^{-1} \mathbf{B}_{(i)} .
\end{align*}
Here $\mathbf{P}_0=\mathbf{I}_n, \mathbf{B}_{(i)}=\left(\mathbf{b}_1, \ldots, \mathbf{b}_i\right)^{\top}$ and $\mathbf{b}_i=\left(Y_{i 1}, \ldots, Y_{i n}\right)^{\top}$ denotes the $i$-th row of the matrix $\mathbf{Y}$, that is $\mathbf{b}_i=\mathbf{A u}_i$, and $\mathbf{P}_i=\left(p_{i, k l}\right)$ is a projection matrix. It should be noted that all $\mathbf{B}_{(i)} \mathbf{B}_{(i)}^{\top}$ are a non-singular matrices with overwhelming probability (see also \cite{gusakova2023volume}). It is also easy to check that $\mathbf{P}_i=\mathbf{P}_i^2$ and that with probability one $\operatorname{tr}\left(\mathbf{P}_i\right)=n-i$. Using \eqref{eq_defUN}, we can write $\mathbf{P}_i$ as
\begin{equation}\label{eq_defPi}
\mathbf{P}_i=\mathbf{I}_n-\mathbf{A} \mathbf{N}_{(i)}^{\top}(\mathbf{N}_{(i)} \mathbf{A}^2 \mathbf{N}_{(i)}^{\top})^{-1} \mathbf{N}_{(i)} \mathbf{A}=(p_{i,k\ell}),
\end{equation}
where $\mathbf{N}_{(i)}$ is the $i \times n$ matrix comprised of the first $i$ rows of $\mathbf{N}$. Similar to \cite{gusakova2023volume},
we define, for $0 \leq i \leq p-1$,
\begin{equation}\label{eq_defTQ}
T_i:=\operatorname{tr}\left(\mathbf{A}^2 \mathbb{E}\left[\mathbf{P}_i\right]\right)=\sum_{k=1}^n \mathbb{E}\left[p_{i, k k}\right] A_{k k}^2 \quad \text { and } \quad \mathbf{Q}_i:=\frac{\mathbf{A} \mathbf{P}_i \mathbf{A}}{T_i}=\left(q_{i, k l}\right),
\end{equation}
such that $\mathbb{E}\left[\operatorname{tr}\left(\mathbf{Q}_i\right)\right]=T_i^{-1} \operatorname{tr}\left(\mathbf{A}^2 \mathbb{E}\left[\mathbf{P}_i\right]\right)=1$. By Assumption \ref{assump1} and the fact $\sum_{k=1}^n p_{i, k k}=n-i$, we have
\begin{align*}
C^{-1}(n-i) \leq T_i \leq C(n-i).
\end{align*}
Finally, we borrow some results from \cite{gusakova2023volume}, which are used frequently in the sequel. We remark here that all the results hold under our settings.
\begin{lemma}[Lemma 3.2 of \cite{gusakova2023volume}]\label{lemma_traceQi}
    For $0\le i\le p-1$, we have $\|\Lambda_{\mathbf{Q}_i}\|\le CT_i^{-1}$,
    \begin{equation}\label{eq_traceQi}
        \operatorname{tr}(\Lambda_{\mathbf{Q}_i}^2)=\frac{1}{T_i^2}\operatorname{tr}(\mathbf{A}^4\mathbf{P}_i)~\text{and}~\frac{C^{-j-1}}{(n-j)^{j-1}}\le \operatorname{tr}(\Lambda_{\mathbf{Q}_i}^j)\le \frac{C^{j+1}}{(n-j)^{j-1}},~j\ge 1.
    \end{equation}
\end{lemma}
Moreover, for given $1\le i\le p-1$ and $1\le k\le n$, we provide the concentration results for $p_{i,kk}$, which is the key ingredient to prove Proposition \ref{lemma_traceQ} below with the proof deferring to Section \ref{proof_pikk}.
\begin{lemma}\label{lemma_pikk}
    Under the conditions of Theorem \ref{thm_logdet}, for given $1\le i\le p-1$ and $1\le k\le n$, it holds that, for any integer $r\ge 1$, 
    \begin{equation}\label{eq_pikk2r}
        \mathbb{E}(p_{i,kk}-\mathbb{E}p_{i,kk})^{2r}\lesssim O(n^{-r})
    \end{equation}
    and 
    \begin{equation}\label{eq_pikk22r}
        \mathbb{E}(p_{i,kk}^2-\mathbb{E}p_{i,kk}^2)^{2r}\lesssim O(n^{-r}).
    \end{equation}
    For $k\ne \ell$, we have
    \begin{equation}\label{eq_pikkpill}
    \begin{split}
        \mathbb{E}(p_{i,kk}p_{i,\ell\ell}-\mathbb{E}p_{i,kk}p_{i,\ell\ell})^{2r}\lesssim O(n^{-r}),
    \end{split}
\end{equation}
\end{lemma}

\subsection{Convergence of $\mathbf{Q}_i$}\label{sec_proof2}
A key input to improve \cite{gusakova2023volume} is the concentration results for the normalized projection matrices $\mathbf{Q}_i$, which is a generalization of Lemma A.1 of \cite{gusakova2023volume}. Throughout this section, we will assume the conditions of Theorem \ref{thm_logdet} and use the notation above. Therefore, we aim to show
\begin{lemma}\label{lemma_traceQ}
    Under the conditions of Theorem \ref{thm_logdet}, it holds that, as $n\rightarrow\infty$,
    \begin{equation}\label{eq_trDQ2}
        \sum_{i=0}^{p-1}\big(\operatorname{tr}(\Lambda_{\mathbf{Q}_i}^2)-\mathbb{E}\operatorname{tr}(\Lambda_{\mathbf{Q}_i}^2)\big)\stackrel{\mathbb{P}}{\rightarrow}0,
    \end{equation}
    \begin{equation}\label{eq_trQ1}
        \sum_{i=0}^{p-1}(\operatorname{tr}\mathbf{Q}_i-1)\stackrel{\mathbb{P}}{\rightarrow}0~\text{and}~\sum_{i=0}^{p-1}\operatorname{Var}(\operatorname{tr}\mathbf{Q}_i)=o(\frac{1}{n}),
    \end{equation}
    \begin{equation}\label{eq_trQ2}
        \sum_{i=0}^{p-1}\big((\operatorname{tr}\mathbf{Q}_i)^2-\mathbb{E}[(\operatorname{tr}\mathbf{Q}_i)^2]\big)\stackrel{\mathbb{P}}{\rightarrow} 0,
    \end{equation}
    where the matrix $\mathbf{Q}_i$ is defined by \eqref{eq_defTQ} and $\Lambda_{\mathbf{Q}_i}$ is defined as the diagonal matrix consists of the eigenvalues of $\mathbf{Q}_i$. Moreover, it holds for $\epsilon>0$ that 
    \begin{equation}\label{eq_maxtrQ1}
        \mathbb{P}(\max_{0\le i\le p-1}|\operatorname{tr}\mathbf{Q}_i-1|>\epsilon)=o(\frac{1}{n^4\epsilon^4}).
    \end{equation}
\end{lemma}
\begin{proof}
Consider \eqref{eq_trDQ2} first. By \eqref{eq_traceQi}, we obtain
\begin{equation*}
    \operatorname{tr}(\Lambda_{\mathbf{Q}_i}^2)-\mathbb{E}\operatorname{tr}(\Lambda_{\mathbf{Q}_i}^2)=\frac{1}{T_i^2}\operatorname{tr}(\mathbf{A}^4(\mathbf{P}_i-\mathbb{E}\mathbf{P}_i))=\frac{1}{T_i^2}\sum_{k=1}^{n}A_{kk}^4(p_{i,kk}-\mathbb{E}p_{i,kk}).
\end{equation*}
The case for $i=0$ does not contribute since $\mathbf{P}_0=\mathbf{I}_0$ and we thus start the summation with $i=1$.
Applying Markov’s and Lyapunov’s inequalities yields for any $\epsilon>0$
\begin{equation*}
    \begin{split}
        &\mathbb{P}\big(|\sum_{i=1}^{p-1}\frac{1}{T_i^2}\sum_{k=1}^{n}A_{kk}^4(p_{i,kk}-\mathbb{E}p_{i,kk})|>\epsilon\big)\le \frac{1}{\epsilon}\sum_{i=1}^{p-1}\frac{1}{T_i^2}\sum_{k=1}^{n}A_{kk}^4\mathbb{E}|p_{i,kk}-\mathbb{E}p_{i,kk}|\\
        \lesssim & \frac{1}{\epsilon}\sum_{i=1}^{p-1}\frac{1}{T_i^2}\sum_{k=1}^{n}A_{kk}^4\big(\mathbb{E}|p_{i,kk}-\mathbb{E}p_{i,kk}|^2\big)^{1/2}\lesssim \frac{1}{n^{1/2}\epsilon}\rightarrow 0,
    \end{split}
\end{equation*}
as $n\rightarrow \infty$, where we used $C^{-1}(n-i)\le T_i\le C(n-i)$ and $\mathbb{E}(p_{i,kk}-\mathbb{E}p_{i,kk})^2\le O(\frac{1}{n})$ for any $1\le i\le p-1$ and $1\le k\le n$, finishing the proof of \eqref{eq_trDQ2}. Next, we turn to \eqref{eq_trQ1}. Notice the definition of $\mathbf{Q}_i$ by \eqref{eq_defTQ} and $\mathbb{E}\operatorname{tr}\mathbf{Q}_i=1$ by normalization, and thus we can write 
\begin{equation*}
    \operatorname{tr}\mathbf{Q}_i-\mathbb{E}\operatorname{tr}\mathbf{Q}_i=\frac{1}{T_i}\sum_{k=1}^{n}A_{kk}^2(p_{i,kk}-\mathbb{E}p_{i,kk}).
\end{equation*}
Setting $a_k=A_{kk}^2-\operatorname{tr}(\mathbf{A}^2)/n$, it is obvious that $\sum_{k=1}^{n}a_k=0$ and 
\begin{equation*}
    \sum_{k=1}^{n}a_{k}^2=\operatorname{tr}\big(\mathbf{A}^2-\frac{\operatorname{tr}\mathbf{A}^2}{n}\mathbf{I}_n\big)^2\lesssim n.
\end{equation*}
Noting $\operatorname{tr}\mathbf{P}_i=n-i$, we can rewrite
\begin{equation*}
    \operatorname{tr}\mathbf{Q}_i-\mathbb{E}\operatorname{tr}\mathbf{Q}_i=\frac{1}{T_i}\sum_{k=1}^{n}A_{kk}^2(p_{i,kk}-\mathbb{E}p_{i,kk})=\frac{1}{T_i}\sum_{k=1}^{n}a_{k}(p_{i,kk}-\mathbb{E}p_{i,kk}).
\end{equation*}
It follows that 
\begin{equation*}
    \begin{split}
        &\mathbb{E}\big(\sum_{i=1}^{p-1}(\operatorname{tr}\mathbf{Q}_i-1)\big)^2=\mathbb{E}\big(\sum_{i=1}^{p-1}\frac{1}{T_i}\sum_{k=1}^{n}a_{k}(p_{i,kk}-\mathbb{E}p_{i,kk})\big)^2\\
        =&\sum_{i=1}^{p-1}\frac{1}{T_i^2}\mathbb{E}\big(\sum_{k=1}^{n}a_{k}(p_{i,kk}-\mathbb{E}p_{i,kk})\big)^2+\sum_{i\ne j}^{p-1}\frac{1}{T_iT_j}\mathbb{E}\big(\sum_{k=1}^{n}a_{k}(p_{i,kk}-\mathbb{E}p_{i,kk})\big)\big(\sum_{k=1}^{n}a_{k}(p_{j,kk}-\mathbb{E}p_{j,kk})\big)\\
        =&: \mathrm{I}+\mathrm{II}.
    \end{split}
\end{equation*}
For the first part, by $T_i\asymp (n-i)^{-1}$ and $\mathbb{E}(p_{i,kk}-\mathbb{E}p_{i,kk})^2=O(n^{-1})$, we get
\begin{equation*}
    \begin{split}
        \mathrm{I}\lesssim &\frac{1}{n}\big(\sum_{k=1}^{n}a_k^2\mathbb{E}(p_{i,kk}-\mathbb{E}p_{i,kk})^2+\sum_{k\ne \ell}^{n}a_{k}a_{\ell}\mathbb{E}(p_{i,kk}-\mathbb{E}p_{i,kk})(p_{i,\ell\ell}-\mathbb{E}p_{i,\ell\ell})\big)\\
        \lesssim &\frac{1}{n^2}\big(\sum_{k=1}^{n}a_k^2+\sum_{k\ne \ell}^{n}|a_ka_{\ell}|\big) \le o(\frac{1}{n}),
    \end{split}
\end{equation*}
where we used \eqref{eq_condition} and 
\begin{equation*}
    |\mathbb{E}(p_{i,kk}-\mathbb{E}p_{i,kk})(p_{i,\ell\ell}-\mathbb{E}p_{i,\ell\ell})|\le \big(\mathbb{E}(p_{i,kk}-\mathbb{E}p_{i,kk})^2\mathbb{E}(p_{i,\ell\ell}-\mathbb{E}p_{i,\ell\ell})^2\big)^{1/2}\lesssim O(n^{-1})
\end{equation*}
by Cauchy-Schwarz inequality and Lemma \ref{lemma_pikk}. For the second part, one has
\begin{equation*}
    \begin{split}
        \mathrm{II}\lesssim &\frac{1}{n^2}\sum_{i\ne j}^{p-1}\big(\sum_{k=1}^{n}a_k^2\mathbb{E}(p_{i,kk}-\mathbb{E}p_{i,kk})(p_{j,kk}-\mathbb{E}p_{j,kk})+\sum_{k\ne \ell}^{n}a_{k}a_{\ell}\mathbb{E}(p_{i,kk}-\mathbb{E}p_{i,kk})(p_{j,\ell\ell}-\mathbb{E}p_{j,\ell\ell})\big)\\
        \lesssim & \frac{1}{n}\big(\sum_{k=1}^{n}a_k^2+\sum_{k\ne \ell}^{n}|a_ka_{\ell}|\big) \le o(1)
    \end{split}
\end{equation*}
by \eqref{eq_condition} in Assumption \ref{assump1}.
Therefore, applying Markov's inequality gives the first part of \eqref{eq_trQ1} as $n\rightarrow\infty$. Moreover, this also gives 
\begin{equation*}
    \sum_{i=1}^{p-1}\operatorname{Var}(\operatorname{tr}\mathbf{Q}_i)=\sum_{i=1}^{p-1}\mathbb{E}(\operatorname{tr}\mathbf{Q}_i-\mathbb{E}\operatorname{tr}\mathbf{Q}_i)^2=\sum_{i=1}^{p-1}\frac{1}{T_i^2}\mathbb{E}\big(\sum_{k=1}^{n}a_{k}(p_{i,kk}-\mathbb{E}p_{i,kk})\big)^2\lesssim o(\frac{1}{n}),
\end{equation*}
which implies the second part of \eqref{eq_trQ1}.

Moreover, applying the Cauchy-Schwarz inequality, for any integers $k\ne \ell\ne r\ne s$, we get
\begin{equation*}
    \begin{split}
        &\mathbb{E}(p_{i,kk}-\mathbb{E}p_{i,kk})^2(p_{i,\ell\ell}-\mathbb{E}p_{i,\ell\ell})^2\lesssim \big(\mathbb{E}(p_{i,kk}-\mathbb{E}p_{i,kk})^4\mathbb{E}(p_{i,\ell\ell}-\mathbb{E}p_{i,\ell\ell})^4\big)^{1/2}\le \frac{1}{n^2},\\
        &\mathbb{E}(p_{i,kk}-\mathbb{E}p_{i,kk})^3(p_{i,\ell\ell}-\mathbb{E}p_{i,\ell\ell})^1\lesssim \big(\mathbb{E}(p_{i,kk}-\mathbb{E}p_{i,kk})^6\mathbb{E}(p_{i,\ell\ell}-\mathbb{E}p_{i,\ell\ell})^2\big)^{1/2}\le \frac{1}{n^2},\\
       &\mathbb{E}(p_{i,kk}-\mathbb{E}p_{i,kk})^2(p_{i,\ell\ell}-\mathbb{E}p_{i,\ell\ell})(p_{i,rr}-\mathbb{E}p_{i,rr})\\
       \lesssim &\big(\mathbb{E}(p_{i,kk}-\mathbb{E}p_{i,kk})^4\mathbb{E}(p_{i,\ell\ell}-\mathbb{E}p_{i,\ell\ell})^2(p_{i,rr}-\mathbb{E}p_{i,rr})^2\big)^{1/2}\le \frac{1}{n^2},
    \end{split}
\end{equation*}
and 
\begin{equation*}
    \begin{split}
        &\mathbb{E}(p_{i,kk}-\mathbb{E}p_{i,kk})(p_{i,\ell\ell}-\mathbb{E}p_{i,\ell\ell})(p_{i,rr}-\mathbb{E}p_{i,rr})(p_{i,ss}-\mathbb{E}p_{i,ss})\\
        \lesssim &\big(\mathbb{E}(p_{i,kk}-\mathbb{E}p_{i,kk})^2(p_{i,\ell\ell}-\mathbb{E}p_{i,\ell\ell})^2\mathbb{E}(p_{i,rr}-\mathbb{E}p_{i,rr})^2(p_{i,ss}-\mathbb{E}p_{i,ss})^2\big)^{1/2}\le \frac{1}{n^2},\\
    \end{split}
\end{equation*}
where we used \eqref{eq_pikk2r}.
Thus, by some direct algebra, we have
\begin{equation*}
    \begin{split}
        &\mathbb{E}(\operatorname{tr}\mathbf{Q}_i-\mathbb{E}\operatorname{tr}\mathbf{Q}_i)^4= \frac{1}{T_i^4}\mathbb{E}\big(\sum_{k=1}^{n}a_{k}(p_{i,kk}-\mathbb{E}p_{i,kk})\big)^4\\
        \lesssim & \frac{1}{(n-i)^4n^2}\big(\sum_{k=1}^{n}a_k^4+\sum_{k\ne \ell}^{n}(a_k^2a_{\ell}^2+|a_k^3a_{\ell}|)+\sum_{k\ne \ell\ne r}|a_k^2a_{\ell}a_{r}|+\sum_{k\ne \ell\ne r\ne s}|a_ka_{\ell}a_{r}a_{s}|\big)\\
        \lesssim & \frac{1}{n^6}\big(\sum_{k=1}^{n}a_k^2+\sum_{k\ne \ell}^{n}(|a_ka_{\ell}|+|a_ka_{\ell}|)+\sum_{k\ne \ell\ne r}|a_k^2a_{\ell}a_{r}|+(\sum_{k\ne \ell}|a_ka_{\ell}|)^2\big)\\
        \lesssim & o(\frac{1}{n^4})
    \end{split}
\end{equation*}
by $|a_k|\le C$ and Assumptions \ref{assump0} and \ref{assump1}, which together with Boole's inequality and Markov's inequality gives, for any $\epsilon>0$, that
\begin{equation*}
    \mathbb{P}(\max_{0\le i\le p-1}|\operatorname{tr}\mathbf{Q}_i-1|>\epsilon)\le \sum_{i=1}^{p-1}\mathbb{P}(|\operatorname{tr}\mathbf{Q}_i-1|>\epsilon)\le \sum_{i=1}^{p-1}\frac{\mathbb{E}(\operatorname{tr}\mathbf{Q}_i-\mathbb{E}\operatorname{tr}\mathbf{Q}_i)^4}{\epsilon^4}\lesssim o(\frac{1}{n^4\epsilon^4}),
\end{equation*}
which also implies that 
\begin{equation*}
    \max_{0\le i\le p-1}|\operatorname{tr}\mathbf{Q}_i-1|\stackrel{\mathbb{P}}{\rightarrow} 0
\end{equation*}
as $n\rightarrow\infty$. For the last statement, invoking \eqref{eq_defTQ} and following the decomposition (A.7) of \cite{gusakova2023volume} gives 
\begin{equation*}
    \begin{split}
        &\sum_{i=0}^{p-1}\big((\operatorname{tr}\mathbf{Q}_i)^2-\mathbb{E}[(\operatorname{tr}\mathbf{Q}_i)^2]\big)\\
        =&\sum_{i=0}^{p-1}\frac{1}{T_i^2}\big(\sum_{k=1}^nA_{kk}^4(p_{i,kk}^2-\mathbb{E}p_{i,kk}^2)+\sum_{k\ne \ell}^nA_{kk}^2A_{\ell\ell}^2(p_{i,kk}p_{i,\ell\ell}-\mathbb{E}p_{i,kk}p_{i,\ell\ell})\big)\\
        =&\sum_{i=0}^{p-1}\frac{1}{T_i^2}\big(\sum_{k=1}^na_{k}^2(p_{i,kk}^2-\mathbb{E}p_{i,kk}^2)+\sum_{k\ne \ell}^na_{k}a_{\ell}(p_{i,kk}p_{i,\ell\ell}-\mathbb{E}p_{i,kk}p_{i,\ell\ell})\big)\\
        &+2\frac{\operatorname{tr}(\mathbf{A}^2)}{n}\sum_{i=0}^{p-1}\frac{n-i}{T_i}(\operatorname{tr}\mathbf{Q}_i-1),
    \end{split}
\end{equation*}
where the second part converges to zero in probability as $n\rightarrow\infty$ by \eqref{eq_trQ1} and the fact that $\operatorname{tr}(\mathbf{A}^2)(n-i)/(nT_i)\asymp 1$. Moreover, by \eqref{eq_pikk22r} and \eqref{eq_pikkpill}, we have
\begin{equation*}
    \begin{split}
        &\sum_{i=0}^{p-1}\frac{1}{T_i^2}\big(\sum_{k=1}^na_{k}^2\mathbb{E}|p_{i,kk}^2-\mathbb{E}p_{i,kk}^2|+\sum_{k\ne \ell}^n|a_{k}a_{\ell}|\mathbb{E}|p_{i,kk}p_{i,\ell\ell}-\mathbb{E}p_{i,kk}p_{i,\ell\ell}|\big)\\
        \lesssim & \sum_{i=0}^{p-1}\frac{1}{T_i^2}\big(\sum_{k=1}^na_{k}^2(\mathbb{E}|p_{i,kk}^2-\mathbb{E}p_{i,kk}^2|^2)^{1/2}+\sum_{k\ne \ell}^n|a_{k}a_{\ell}|(\mathbb{E}|p_{i,kk}p_{i,\ell\ell}-\mathbb{E}p_{i,kk}p_{i,\ell\ell}|^2)^{1/2}\big)\\
        \lesssim & \frac{1}{n}\big(\sum_{k=1}^na_{k}^2\frac{1}{n^{1/2}}+\sum_{k\ne \ell}^n|a_{k}a_{\ell}|\frac{1}{n^{1/2}}\big)\lesssim o(\frac{1}{n^{1/2}}),
    \end{split}
\end{equation*}
by \eqref{eq_condition}, which further implies \eqref{eq_trQ2} as $n\rightarrow\infty$, completing the proof of this proposition.
\end{proof}

\section{Proof of Lemma \ref{lemma_pikk}}\label{proof_pikk}
This section consists of the proof of Lemma \ref{lemma_pikk}, which is essential to prove Proposition \ref{lemma_traceQ}.
Recalling \eqref{eq_defPi} and \eqref{eq_defTQ}, for the case of $\mathbf{A}=\mathbf{I}_n$, it is well-known that $p_{i,kk}$ are  distributed identically, which further implies $\mathbb{E}p_{i,kk}=(n-i)/n$. However, for the general case $\mathbf{A}\ne \mathbf{I}_n$, the estimations for $\mathbb{E}p_{i,kk}$ and $\mathbb{E}(p_{i,kk}-\mathbb{E}p_{i,kk})^2$ are much involved. Denote the $k$-th column of $\mathbf{N}_{(i)}$ by $\mathbf{w}_{i k}$ and set $\mathbf{B}_{(i)}=\mathbf{N}_{(i)} \mathbf{A}$. Then the $k$-th column of $\mathbf{B}_{(i)}$ is $\mathbf{v}_{i k}=A_{k k} \mathbf{w}_{i k}$. In view of \eqref{eq_defPi}, we write
\begin{equation}\label{eq_defpikk}
p_{i, k k}=1-\mathbf{v}_{i k}^{\top}(\mathbf{B}_{(i)} \mathbf{B}_{(i)}^{\top})^{-1} \mathbf{v}_{i k}=\frac{1}{1+\mathbf{w}_{i k}^{\top} \mathbf{M}_{i k} \mathbf{w}_{i k}},
\end{equation}
where we used the notation
\begin{equation}\label{eq_defMik}
\mathbf{M}_{i k}:=A_{k k}^2(\sum_{\ell=1 ; \ell \ne k}^n \mathbf{v}_{i \ell} \mathbf{v}_{i \ell}^{\top})^{-1}=A_{k k}^2(\sum_{\ell=1 ; \ell \ne k}^n A_{\ell\ell}^2\mathbf{w}_{i \ell} \mathbf{w}_{i \ell}^{\top})^{-1}=A_{k k}^2(\mathbf{B}_{(i, k)} \mathbf{B}_{(i, k)}^{\top})^{-1}
\end{equation}
with $\mathbf{B}_{(i, k)}$ denoting the matrix obtained from $\mathbf{B}_{(i)}$ by deleting its $k$-th column $\mathbf{v}_{i k}$.
Let $1 \le i \le p-1$ and $1 \le k \le n$, since $\mathbf{P}_0=\mathbf{I}_n$ for the trivial case $i=0$. The concentration for $p_{i,kk}$ relies on the estimations of $\mathbb{E}\operatorname{tr}\mathbf{M}_{ik}^s$ for some integer $s>0$, which does not have an explicit form for general $\mathbf{A}\ne \mathbf{I}$. However, one can expect that $\mathbf{M}_{ik}$ behaves as $(\sum_{\ell=1,\ell\ne k}^{n}\mathbf{w}_{i\ell}\mathbf{w}_{i\ell}^{\top})^{-1}$ similarly, which is the inverse of a Wishart matrix. We start with a recursive moments formula for the inverse Wishart matrices.
\begin{lemma}[Theorem 3.1 of \cite{pielaszkiewicz2020mixtures}]\label{lemma_momentWishart}
Let $\mathbf{W} \sim \text{Wishart}~(\mathbf{I}_p, n)$. Then, the following recursive formula holds for all $k \in \mathbb{N}$ and all $m_0, m_1, \ldots, m_k$ such that $m_0=0, m_k \in \mathbb{N}, m_i \in \mathbb{N}_0, i=1, \ldots, k-1$
\begin{align*}
\begin{aligned}
&\left(n-p-m_k\right) \mathbb{E}\left[\prod_{i=0}^k \operatorname{tr}\left\{\mathbf{W}^{-m_i}\right\}\right]\\
= & \mathbb{E}\left[\operatorname{tr}\left\{\mathbf{W}^{-m_k+1}\right\} \prod_{i=0}^{k-1} \operatorname{tr}\left\{\mathbf{W}^{-m_i}\right\}\right] \\
& +2 \sum_{i=0}^{k-1} m_i \mathbb{E}\left[\operatorname{tr}\left\{\mathbf{W}^{-m_k-m_i}\right\} \prod_{\substack{j=0 \\
j \ne i}}^{k -1} \operatorname{tr}\left\{\mathbf{W}^{-m_j}\right\}\right] \\
& +\sum_{i=0}^{m_k-2} \mathbb{E}\left[\operatorname{tr}\left\{\mathbf{W}^{-i-1}\right\} \operatorname{tr}\left\{\mathbf{W}^{-m_k+1+i}\right\} \prod_{j=0}^{k-1} \operatorname{tr}\left\{\mathbf{W}^{-m_j}\right\}\right].
\end{aligned}
\end{align*}
\end{lemma}
In view of that we can write the Wishart matrices $\mathbf{W}=\sum_{\ell=1}^{n}\mathbf{w}_{i,\ell}\mathbf{w}_{i,\ell}^{\top}\sim \text{Wishart}~(\mathbf{I}_i,n)$, consequently, we have 
\begin{equation*}
    \mathbb{E}\operatorname{tr}\mathbf{W}=\frac{i}{n-i-1}.
\end{equation*}
by choosing $m_1=1$ and $k=1$. Moreover, setting $m_1=m_2=1$ for $k=2$ and $m_1=2$ for $k=1$ gives
\begin{equation*}
    \begin{split}
        (n-i-1)\mathbb{E}[\operatorname{tr}\mathbf{W}^0\operatorname{tr}\mathbf{W}^{-1}\operatorname{tr}\mathbf{W}^{-1}]=\frac{i^3}{n-i-1}+2\mathbb{E}[\operatorname{tr}\mathbf{W}^{-2}\operatorname{tr}\mathbf{W}^0],\\
        (n-i-2)\mathbb{E}[\operatorname{tr}\mathbf{W}^0\operatorname{tr}\mathbf{W}^{-2}]=\frac{i^2}{n-i-1}+\mathbb{E}[\operatorname{tr}\mathbf{W}^0\operatorname{tr}\mathbf{W}^{-1}\operatorname{tr}\mathbf{W}^{-1}],
    \end{split}
\end{equation*}
which further implies
\begin{equation*}
        \mathbb{E}[(\operatorname{tr}(\mathbf{W}^{-1})^2]=\frac{i[i(n-i-2)+2]}{(n-i)(n-i-1)(n-i-3)}
    \end{equation*}
    and
\begin{equation*}
        \mathbb{E}\operatorname{tr}(\mathbf{W}^{-2})=\frac{i(n-1)}{(n-i)(n-i-1)(n-i-3)}.
\end{equation*}
For a general integer $K\ge 2$, one can use the recursive formula to get 
\begin{equation}\label{eq_meanWk}
    \mathbb{E}\operatorname{tr}(\mathbf{W}^{-K})\lesssim \frac{i}{n^{K-1}(n-i)}~\text{and}~\mathbb{E}[(\operatorname{tr}\mathbf{W}^{-1})^{K}]\lesssim \frac{i^K}{(n-i)^K}
\end{equation}
since $\|\mathbf{W}\|\ge cn$ for some constant $c>0$ as $n\rightarrow \infty$ under Assumptions \ref{assump0} and \ref{assump1}, which further implies 
\begin{equation}\label{eq_meanWk2}
   \mathbb{E}[(\operatorname{tr}\mathbf{W}^{-2})^{K}]\lesssim \frac{1}{n^{K}}\mathbb{E}[(\operatorname{tr}\mathbf{W}^{-1})^{K}]\lesssim \frac{i^K}{[n(n-i)]^K}.
\end{equation}

Thus, we have the following result for $\mathbf{M}_{ik}$.
\begin{lemma}\label{lemma_EtrMik}
    Under the conditions of Theorem \ref{thm_logdet}, we have, for fixed integer $r>0$ 
    \begin{equation}\label{eq_boundsEtrMik}
        C^{-1}\frac{i}{n^{r-1}(n-i)}\le \mathbb{E}\operatorname{tr}\mathbf{M}_{ik}^r\le C\frac{i}{n^{r-1}(n-i)}
    \end{equation}
    and
     \begin{equation}\label{eq_boundsEtrMik2}
        C^{-1}\frac{i^r}{(n(n-i))^{r}}\le \mathbb{E}(\operatorname{tr}\mathbf{M}_{ik}^2)^r\le C\frac{i^r}{(n(n-i))^{r}}
    \end{equation}
    for some constant $C>1$. 
\end{lemma}
\begin{proof}
    Note the Loewner order, which states that if $0\prec \mathbf{D}_1\preceq \mathbf{D}_2$, then $\mathbf{D}_2^{-1}\preceq \mathbf{D}_1^{-1}$. By Assumption \ref{assump1}, we define 
    \begin{equation*}
        \mathbf{L}_{ik,\min}=\min_{\ell\ne k}A_{\ell\ell}^2 \sum_{\ell=1,\ell\ne k}^{n}\mathbf{w}_{i \ell} \mathbf{w}_{i \ell}^{\top}~\text{and}~\mathbf{L}_{ik,\max}=\max_{\ell\ne k}A_{\ell\ell}^2 \sum_{\ell=1,\ell\ne k}^{n}\mathbf{w}_{i \ell} \mathbf{w}_{i \ell}^{\top},
    \end{equation*}
    which further gives
    \begin{equation*}
        0\prec \mathbf{L}_{ik,\min}\preceq A_{kk}^2\mathbf{M}_{ik}^{-1}\preceq \mathbf{L}_{ik,\max}.
    \end{equation*}
    Thus, applying the inverse monotonicity and noting the trace is a linear function that preserves the Loewner order on positive definite matrices, we get, for any positive integer $r$,
    \begin{equation*}
       \operatorname{tr}(\mathbf{L}_{ik,\max}^{-r})\le A_{kk}^{-2r}\operatorname{tr}\mathbf{M}_{ik}^r\le \operatorname{tr}(\mathbf{L}_{ik,\min}^{-r}),
    \end{equation*}
    which gives the desired results by the linearity of expectation and \eqref{eq_meanWk} and \eqref{eq_meanWk2}, as well as the Assumptions \ref{assump0} and \ref{assump1}.
\end{proof}
Next, we present a lemma for the concentration on quadratic forms; see Lemma B.26 of \cite{bai2010spectral} for details.
\begin{lemma}[Lemma B.26 of \cite{bai2010spectral}]\label{lemma_quadratic}
    Let $\mathbf{B}$ be an $n\times n$ non-random symmetric matrix with bounded spectral norm and $\mathbf{w}\in \mathbb{R}^{n}$ with i.i.d. standard Gaussian normal random variables. Then for any $s\ge 1$,
    \begin{equation*}
        \mathbb{E}|\mathbf{w}^{\top}\mathbf{B}\mathbf{w}-\operatorname{tr}\mathbf{B}|^s\le C_s(\operatorname{tr}\mathbf{B}^s+(\operatorname{tr}\mathbf{B}^2)^{s/2}). 
    \end{equation*}
\end{lemma}

With Lemmas \ref{lemma_EtrMik} and \ref{lemma_quadratic}, we have
\begin{lemma}\label{lemma_qudra}
    Under the conditions of Theorem \ref{thm_logdet}, we have, for fixed integer $r>0$
    \begin{equation}\label{eq_wMikwr}
    \mathbb{E}(w^{\top}\mathbf{M}_{ik}w-\operatorname{tr}\mathbf{M}_{ik})^{2r}
        \lesssim \frac{i^r}{n^r(n-i)^{r}}
\end{equation}
and 
\begin{equation}\label{eq_trMik2r}
    \mathbb{E}(\operatorname{tr} \mathbf{M}_{i k}-\mathbb{E} \operatorname{tr} \mathbf{M}_{i k})^{2r}\lesssim O(n^{-r}).
\end{equation}
\end{lemma}
\begin{proof}
    Consider \eqref{eq_wMikwr} first. By the condition expectation argument, applying Lemma \ref{lemma_quadratic} and Lemma \ref{lemma_EtrMik} gives
\begin{equation*}
    \mathbb{E}(w^{\top}\mathbf{M}_{ik}w-\operatorname{tr}\mathbf{M}_{ik})^{2r}
        \lesssim  \mathbb{E}[\operatorname{tr}(\mathbf{M}_{ik}^{2r})+(\operatorname{tr}\mathbf{M}_{ik}^2)^{r}]\lesssim \frac{i^r}{n^r(n-i)^{r}}
\end{equation*}
for general integers $r\ge 1$ as $n\rightarrow \infty$, finishing the proof of \eqref{eq_wMikwr}. It suffices to show \eqref{eq_trMik2r}, which can be handled by the martingale decomposition \cite{bai2010spectral}. Specifically, let $\mathbb{E}_{\ell}[\cdot]:=\mathbb{E}\left[\cdot \mid\left\{v_{i \ell}, \ldots, v_{i n}\right\} \backslash\left\{v_{i k}\right\}\right]$ for $\ell=1, \ldots, n$ and $\mathbb{E}_{n+1}:=\mathbb{E}$. It holds
\begin{align*}
\begin{aligned}
&\operatorname{tr} \mathbf{M}_{i k}-\mathbb{E} \operatorname{tr} \mathbf{M}_{i k}=\sum_{\ell=1 ; \ell \ne k}^n\left(\mathbb{E}_{\ell}-\mathbb{E}_{\ell+1}\right) \operatorname{tr} \mathbf{M}_{i k} \\
= & A_{k k}^2 \sum_{\ell=1 ; \ell \ne k}^n\left(\mathbb{E}_{\ell}-\mathbb{E}_{\ell+1}\right)\big(\operatorname{tr}(\mathbf{B}_{(i, k)} \mathbf{B}_{(i, k)}^{\top})^{-1}-\operatorname{tr}(\mathbf{B}_{(i, k)} \mathbf{B}_{(i, k)}^{\top}-\mathbf{v}_{i \ell} \mathbf{v}_{i \ell}^{\top})^{-1}\big),
\end{aligned}
\end{align*}
where we used the facts that $\left(\mathbb{E}_k-\mathbb{E}_{k+1}\right) \operatorname{tr} \mathbf{M}_{i k}=0$ and $\left(\mathbb{E}_{\ell}-\mathbb{E}_{\ell+1}\right) \operatorname{tr}(\mathbf{B}_{(i, k)}\mathbf{B}_{(i, k)}^{\top}-\mathbf{v}_{i \ell} \mathbf{v}_{i \ell}^{\top})^{-1}=0$ for $\ell \ne k$. Applying the Burkholder inequality for martingale differences (Lemma 2.12 of \cite{bai2010spectral}), we deduce
\begin{align*}
\begin{aligned}
&\mathbb{E}\left|\operatorname{tr} \mathbf{M}_{i k}-\mathbb{E} \operatorname{tr} \mathbf{M}_{i k}\right|^2 \\
\lesssim & \sum_{\ell=1 ; \ell \ne k}^n \mathbb{E}\left|\left(\mathbb{E}_{\ell}-\mathbb{E}_{\ell+1}\right)\big(\operatorname{tr}(\mathbf{B}_{(i,k)} \mathbf{B}_{(i,k)}^{\top})^{-1}-\operatorname{tr}(\mathbf{B}_{(i,k)} \mathbf{B}_{(i,k)}^{\top}-\mathbf{v}_{i \ell} \mathbf{v}_{i \ell}^{\top})^{-1}\big)\right|^2 \\
\lesssim & \sum_{\ell=1 ; \ell \ne k}^n \mathbb{E}\left|\left(\mathbb{E}_{\ell}-\mathbb{E}_{\ell+1}\right) \frac{\mathbf{v}_{i \ell}^{\top}(\mathbf{B}_{(i,k)} \mathbf{B}_{(i,k)}^{\top}-\mathbf{v}_{i \ell} \mathbf{v}_{i \ell}^{\top})^{-2} \mathbf{v}_{i \ell}}{1+\mathbf{v}_{i \ell}^{\top}(\mathbf{B}_{(i,k)} \mathbf{B}_{(i,k)}^{\top}-\mathbf{v}_{i \ell} \mathbf{v}_{i \ell}^{\top})^{-1} \mathbf{v}_{i \ell}}\right|^2 \\
\lesssim & \frac{1}{n^2} \sum_{\ell=1 ; \ell \ne k}^n \mathbb{E}\left|\left(\mathbb{E}_{\ell}+\mathbb{E}_{\ell+1}\right) \frac{v_{i \ell}^{\top}(\mathbf{B}_{(i,k)} \mathbf{B}_{(i,k)}^{\top}-\mathbf{v}_{i \ell} \mathbf{v}_{i \ell}^{\top})^{-1} \mathbf{v}_{i \ell}}{1+\mathbf{v}_{i \ell}^{\top}(\mathbf{B}_{(i,k)} \mathbf{B}_{(i,k)}^{\top}-\mathbf{v}_{i \ell} \mathbf{v}_{i \ell}^{\top})^{-1} \mathbf{v}_{i \ell}}\right|^2 \\
\le & \frac{2^2(n-1)}{n^2}=O(n^{-1}), \quad n \rightarrow \infty,
\end{aligned}
\end{align*}
where for the third line we used the Sherman-Morrison formula and for the fourth line the fact that the spectral norms $n\|(\mathbf{B}_{(i,k)} \mathbf{B}_{(i,k)}^{\top}-\mathbf{v}_{i \ell} \mathbf{v}_{i \ell}^{\top})^{-1}\|$ are uniformly bounded by a constant. For general integers $r\ge 2$, due to the definition of martingale differences, all the terms in the expansion with at least one will vanish, and thus the non-zero terms in the summation at most be $n^{r}$, which implies \eqref{eq_trMik2r} as $n\rightarrow \infty$, completing the proof of this proposition.
\end{proof}

We now turn to the proof of Lemma \ref{lemma_pikk}. To simplify notation we will write $w$ instead of $\mathbf{w}_{i k}$. Multiple applications of the identity
\begin{align*}
\frac{1}{1+x}-\frac{1}{1+y}=\frac{y-x}{(1+x)(1+y)}, \quad x, y \geq 0
\end{align*}
give
\begin{align*}
\begin{aligned}
p_{i, k k}- \mathbb{E}\left[p_{i, k k}\right]=&\frac{1}{1+w^{\top} \mathbf{M}_{i k} w}-\mathbb{E}\left[\frac{1}{1+w^{\top} \mathbf{M}_{i k} w}\right] \\
= & -\frac{1}{1+\operatorname{tr} \mathbf{M}_{i k}}+\frac{1}{1+w^{\top} \mathbf{M}_{i k} w}+\mathbb{E}\left[\frac{1}{1+\mathbb{E}\left[\operatorname{tr} \mathbf{M}_{i k}\right]}\right]-\mathbb{E}\left[\frac{1}{1+w^{\top} \mathbf{M}_{i k} w}\right] \\
& -\frac{1}{1+\mathbb{E}\left[\operatorname{tr} \mathbf{M}_{i k}\right]}+\frac{1}{1+\operatorname{tr} \mathbf{M}_{i k}} \\
= & -\frac{w^{\top} \mathbf{M}_{i k} w-\operatorname{tr} \mathbf{M}_{i k}}{\left(1+w^{\top} \mathbf{M}_{i k} w\right)\left(1+\operatorname{tr} \mathbf{M}_{i k}\right)}+\mathbb{E}\left[\frac{w^{\top} \mathbf{M}_{i k} w-\mathbb{E}\left[\operatorname{tr} \mathbf{M}_{i k}\right]}{\left(1+w^{\top} \mathbf{M}_{i k} w\right)\left(1+\mathbb{E}\left[\operatorname{tr} \mathbf{M}_{i k}\right]\right)}\right] \\
& -\frac{\operatorname{tr} \mathbf{M}_{i k}-\mathbb{E}\left[\operatorname{tr} \mathbf{M}_{i k}\right]}{\left(1+\operatorname{tr} \mathbf{M}_{i k}\right)\left(1+\mathbb{E}\left[\operatorname{tr} \mathbf{M}_{i k}\right]\right)}\\
=&: -T_{ik,1}+T_{ik,2}-T_{ik,3}.
\end{aligned}
\end{align*}
\textbf{I. The part $T_{ik,2}$}\\
We first note that $T_{ik,2}=\mathbb{E}p_{i,kk}-\frac{1}{1+\mathbb{E}\operatorname{tr}\mathbf{M}_{ik}}$. Thus, by condition argument, we get $\mathbb{E}w^{\top}\mathbf{M}_{ik}w=\mathbb{E}\operatorname{tr}\mathbf{M}_{ik}$, which further gives 
\begin{equation}\label{eq_diffEpikk}
   \begin{split}
       &\mathbb{E}p_{i,kk}-\frac{1}{1+\mathbb{E}\operatorname{tr}\mathbf{M}_{ik}}=-\mathbb{E}\left(\frac{w^{\top}\mathbf{M}_{ik}w-\operatorname{tr}\mathbf{M}_{ik}+\operatorname{tr}\mathbf{M}_{ik}-\mathbb{E}\operatorname{tr}\mathbf{M}_{ik}}{(1+w^{\top}\mathbf{M}_{ik}w)(1+\mathbb{E}\operatorname{tr}\mathbf{M}_{ik})}\right)\\
       =&-\mathbb{E}\left(\frac{w^{\top}\mathbf{M}_{ik}w-\operatorname{tr}\mathbf{M}_{ik}}{(1+w^{\top}\mathbf{M}_{ik}w)(1+\mathbb{E}\operatorname{tr}\mathbf{M}_{ik})}\right)-\mathbb{E}\left(\frac{\operatorname{tr}\mathbf{M}_{ik}-\mathbb{E}\operatorname{tr}\mathbf{M}_{ik}}{(1+w^{\top}\mathbf{M}_{ik}w)(1+\mathbb{E}\operatorname{tr}\mathbf{M}_{ik})}\right).
   \end{split}
\end{equation}
For the first term, direct computation gives
\begin{equation*}
    \begin{split}
        &\frac{w^{\top}\mathbf{M}_{ik}w-\operatorname{tr}\mathbf{M}_{ik}}{(1+w^{\top}\mathbf{M}_{ik}w)}\\
        =&\frac{w^{\top}\mathbf{M}_{ik}w-\operatorname{tr}\mathbf{M}_{ik}}{(1+\mathbb{E}\operatorname{tr}\mathbf{M}_{ik})}-\frac{(w^{\top}\mathbf{M}_{ik}w-\operatorname{tr}\mathbf{M}_{ik})(w^{\top}\mathbf{M}_{ik}w-\operatorname{tr}\mathbf{M}_{ik}+\operatorname{tr}\mathbf{M}_{ik}-\mathbb{E}\operatorname{tr}\mathbf{M}_{ik})}{(1+w^{\top}\mathbf{M}_{ik}w)(1+\mathbb{E}\operatorname{tr}\mathbf{M}_{ik})},
    \end{split}
\end{equation*}
which further implies 
\begin{equation*}
    \begin{split}
        &-\mathbb{E}\frac{w^{\top}\mathbf{M}_{ik}w-\operatorname{tr}\mathbf{M}_{ik}}{(1+w^{\top}\mathbf{M}_{ik}w)}\\
    =&\mathbb{E}\left(\frac{(w^{\top}\mathbf{M}_{ik}w-\operatorname{tr}\mathbf{M}_{ik})^2}{(1+w^{\top}\mathbf{M}_{ik}w)(1+\mathbb{E}\operatorname{tr}\mathbf{M}_{ik})}\right)+\mathbb{E}\frac{(w^{\top}\mathbf{M}_{ik}w-\operatorname{tr}\mathbf{M}_{ik})(\operatorname{tr}\mathbf{M}_{ik}-\mathbb{E}\operatorname{tr}\mathbf{M}_{ik})}{(1+w^{\top}\mathbf{M}_{ik}w)(1+\mathbb{E}\operatorname{tr}\mathbf{M}_{ik})},
    \end{split}
\end{equation*}
since the first part has mean zero by condition argument. Similarly, for the second part, we also have 
\begin{equation*}
    \begin{split}
        &\frac{\operatorname{tr}\mathbf{M}_{ik}-\mathbb{E}\operatorname{tr}\mathbf{M}_{ik}}{(1+w^{\top}\mathbf{M}_{ik}w)}\\
    =&\frac{\operatorname{tr}\mathbf{M}_{ik}-\mathbb{E}\operatorname{tr}\mathbf{M}_{ik}}{(1+\mathbb{E}\operatorname{tr}\mathbf{M}_{ik})}-\frac{(\operatorname{tr}\mathbf{M}_{ik}-\mathbb{E}\operatorname{tr}\mathbf{M}_{ik})(w^{\top}\mathbf{M}_{ik}w-\operatorname{tr}\mathbf{M}_{ik}+\operatorname{tr}\mathbf{M}_{ik}-\mathbb{E}\operatorname{tr}\mathbf{M}_{ik})}{(1+w^{\top}\mathbf{M}_{ik}w)(1+\mathbb{E}\operatorname{tr}\mathbf{M}_{ik})}
    \end{split}
\end{equation*}
and further taking the expectation on both two sides gives
\begin{equation*}
   \begin{split}
        &-\mathbb{E} \frac{\operatorname{tr}\mathbf{M}_{ik}-\mathbb{E}\operatorname{tr}\mathbf{M}_{ik}}{(1+w^{\top}\mathbf{M}_{ik}w)}\\
    =&\mathbb{E}\left(\frac{(\operatorname{tr}\mathbf{M}_{ik}-\mathbb{E}\operatorname{tr}\mathbf{M}_{ik})^2}{(1+w^{\top}\mathbf{M}_{ik}w)(1+\mathbb{E}\operatorname{tr}\mathbf{M}_{ik})}\right)+\mathbb{E}\frac{(w^{\top}\mathbf{M}_{ik}w-\operatorname{tr}\mathbf{M}_{ik})(\operatorname{tr}\mathbf{M}_{ik}-\mathbb{E}\operatorname{tr}\mathbf{M}_{ik})}{(1+w^{\top}\mathbf{M}_{ik}w)(1+\mathbb{E}\operatorname{tr}\mathbf{M}_{ik})},
   \end{split}
\end{equation*}
since the first term has mean zero. 

Notice the trivial bound $w^{\top}\mathbf{M}_{ik}w\ge 0$, and Lemma \ref{lemma_qudra}. Applying the Cauchy-Schwarz inequality gives
\begin{equation*}
   \begin{split}
        &\mathbb{E}(w^{\top}\mathbf{M}_{ik}w-\operatorname{tr}\mathbf{M}_{ik})(\operatorname{tr}\mathbf{M}_{ik}-\mathbb{E}\operatorname{tr}\mathbf{M}_{ik})\\
        \le &\big(\mathbb{E}(w^{\top}\mathbf{M}_{ik}w-\operatorname{tr}\mathbf{M}_{ik})^2\cdot \mathbb{E}(\operatorname{tr} \mathbf{M}_{i k}-\mathbb{E} \operatorname{tr} \mathbf{M}_{i k})^2\big)^{1/2}\\
        \lesssim & C\big(\frac{i}{n(n-i)}\frac{1}{n}\big)^{1/2}\le C\big(\frac{i}{n(n-i)}\big)^{1/2}
   \end{split}
\end{equation*}
which further gives 
\begin{equation}\label{eq_Epikk}
    \mathbb{E}p_{i,kk}=\frac{1}{1+\mathbb{E}\operatorname{tr}\mathbf{M}_{ik}}+O(\frac{1}{n}).
\end{equation}
and thus
\begin{equation}\label{eq_Tik2}
    T_{ik,2}=\mathbb{E}p_{i,kk}-\frac{1}{1+\mathbb{E}\operatorname{tr}\mathbf{M}_{ik}}=O(\frac{1}{n})
\end{equation}
\textbf{II. The part $T_{ik,3}$}\\
Next, in view of the identity 
\begin{equation*}
    \begin{split}
        \frac{1}{1+\operatorname{tr} \mathbf{M}_{i k}}=&\frac{1}{1+\mathbb{E}\operatorname{tr}\mathbf{M}_{ik}}-\frac{(\operatorname{tr}\mathbf{M}_{ik}-\mathbb{E}\operatorname{tr}\mathbf{M}_{ik})}{(1+\operatorname{tr}\mathbf{M}_{ik})(1+\mathbb{E}\operatorname{tr}\mathbf{M}_{ik})}\\
        =&\frac{1}{1+\mathbb{E}\operatorname{tr}\mathbf{M}_{ik}}-\frac{(\operatorname{tr}\mathbf{M}_{ik}-\mathbb{E}\operatorname{tr}\mathbf{M}_{ik})}{(1+\mathbb{E}\operatorname{tr}\mathbf{M}_{ik})^2}+\frac{(\operatorname{tr}\mathbf{M}_{ik}-\mathbb{E}\operatorname{tr}\mathbf{M}_{ik})^2}{(1+\operatorname{tr}\mathbf{M}_{ik})(1+\mathbb{E}\operatorname{tr}\mathbf{M}_{ik})^2},
    \end{split}
\end{equation*}
we can write $T_{ik,3}$ as
\begin{equation*}
    \begin{split}
        T_{ik,3}=&\frac{(\operatorname{tr}\mathbf{M}_{ik}-\mathbb{E}\operatorname{tr}\mathbf{M}_{ik})}{(1+\mathbb{E}\operatorname{tr}\mathbf{M}_{ik})^2}-\frac{(\operatorname{tr}\mathbf{M}_{ik}-\mathbb{E}\operatorname{tr}\mathbf{M}_{ik})^2}{(1+\mathbb{E}\operatorname{tr}\mathbf{M}_{ik})^3}\\
        &+\frac{(\operatorname{tr}\mathbf{M}_{ik}-\mathbb{E}\operatorname{tr}\mathbf{M}_{ik})^3}{(1+\mathbb{E}\operatorname{tr}\mathbf{M}_{ik})^4}+\frac{(\operatorname{tr}\mathbf{M}_{ik}-\mathbb{E}\operatorname{tr}\mathbf{M}_{ik})^4}{(1+\operatorname{tr}\mathbf{M}_{ik})(1+\mathbb{E}\operatorname{tr}\mathbf{M}_{ik})^4}.
    \end{split}
\end{equation*}
Thereafter, by Lemma \ref{lemma_qudra} for any integer $r\ge 1$, we have
 $\mathbb{E}T_{ik,3}=O(\frac{1}{n})$ and 
\begin{equation}\label{eq_Tik3}
    \mathbb{E}\big(T_{ik,3}-\frac{\operatorname{tr}\mathbf{M}_{ik}-\mathbb{E}\operatorname{tr}\mathbf{M}_{ik}}{(1+\mathbb{E}\operatorname{tr}\mathbf{M}_{ik})^2}\big)^{2r}=O(n^{-2r}).
\end{equation}
\textbf{III. The part $T_{ik,1}$}\\
We now turn to $T_{ik,1}$, which is much involved. By the identities
\begin{equation*}
    \frac{1}{\left(1+w^{\top} \mathbf{M}_{i k} w\right)}=
    \frac{1}{1+\mathbb{E}\operatorname{tr}\mathbf{M}_{ik}}-\frac{(w^{\top}\mathbf{M}_{ik}w-\operatorname{tr}\mathbf{M}_{ik}+\operatorname{tr}\mathbf{M}_{ik}-\mathbb{E}\operatorname{tr}\mathbf{M}_{ik})}{(1+w^{\top}\mathbf{M}_{ik}w)(1+\mathbb{E}\operatorname{tr}\mathbf{M}_{ik})}
\end{equation*}
and
\begin{equation*}
    \frac{1}{1+\operatorname{tr} \mathbf{M}_{i k}}=\frac{1}{1+\mathbb{E}\operatorname{tr}\mathbf{M}_{ik}}-\frac{(\operatorname{tr}\mathbf{M}_{ik}-\mathbb{E}\operatorname{tr}\mathbf{M}_{ik})}{(1+\mathbb{E}\operatorname{tr}\mathbf{M}_{ik})^2}+\frac{(\operatorname{tr}\mathbf{M}_{ik}-\mathbb{E}\operatorname{tr}\mathbf{M}_{ik})^2}{(1+\operatorname{tr}\mathbf{M}_{ik})(1+\mathbb{E}\operatorname{tr}\mathbf{M}_{ik})^2},
\end{equation*}
we have
\begin{equation*}
    \begin{split}
        T_{ik,1}=&\frac{w^{\top} \mathbf{M}_{i k} w-\operatorname{tr} \mathbf{M}_{i k}}{\left(1+w^{\top} \mathbf{M}_{i k} w\right)\left(1+\operatorname{tr} \mathbf{M}_{i k}\right)}\\
        =&(w^{\top} \mathbf{M}_{i k} w-\operatorname{tr} \mathbf{M}_{i k})\left(\frac{1}{1+\mathbb{E}\operatorname{tr}\mathbf{M}_{ik}}-\frac{(w^{\top}\mathbf{M}_{ik}w-\operatorname{tr}\mathbf{M}_{ik}+\operatorname{tr}\mathbf{M}_{ik}-\mathbb{E}\operatorname{tr}\mathbf{M}_{ik})}{(1+w^{\top}\mathbf{M}_{ik}w)(1+\mathbb{E}\operatorname{tr}\mathbf{M}_{ik})}\right)\\
        &\times \left(\frac{1}{1+\mathbb{E}\operatorname{tr}\mathbf{M}_{ik}}-\frac{(\operatorname{tr}\mathbf{M}_{ik}-\mathbb{E}\operatorname{tr}\mathbf{M}_{ik})}{(1+\operatorname{tr}\mathbf{M}_{ik})(1+\mathbb{E}\operatorname{tr}\mathbf{M}_{ik})}\right)\\
        =&\frac{(w^{\top} \mathbf{M}_{i k} w-\operatorname{tr} \mathbf{M}_{i k})}{(1+\mathbb{E}\operatorname{tr}\mathbf{M}_{ik})^2}-\frac{(w^{\top} \mathbf{M}_{i k} w-\operatorname{tr} \mathbf{M}_{i k})^2}{(1+w^{\top}\mathbf{M}_{ik}w)(1+\mathbb{E}\operatorname{tr}\mathbf{M}_{ik})^2}\\
        &-2\frac{(w^{\top} \mathbf{M}_{i k} w-\operatorname{tr} \mathbf{M}_{i k})(\operatorname{tr}\mathbf{M}_{ik}-\mathbb{E}\operatorname{tr}\mathbf{M}_{ik})}{(1+\operatorname{tr}\mathbf{M}_{ik})(1+\mathbb{E}\operatorname{tr}\mathbf{M}_{ik})^2}\\
        &+\frac{(w^{\top} \mathbf{M}_{i k} w-\operatorname{tr} \mathbf{M}_{i k})(w^{\top}\mathbf{M}_{ik}w-\operatorname{tr}\mathbf{M}_{ik}+\operatorname{tr}\mathbf{M}_{ik}-\mathbb{E}\operatorname{tr}\mathbf{M}_{ik})(\operatorname{tr}\mathbf{M}_{ik}-\mathbb{E}\operatorname{tr}\mathbf{M}_{ik})}{(1+w^{\top}\mathbf{M}_{ik}w)(1+\operatorname{tr}\mathbf{M}_{ik})(1+\mathbb{E}\operatorname{tr}\mathbf{M}_{ik})^2}.
    \end{split}
\end{equation*}
Applying \eqref{eq_wMikwr}, \eqref{eq_trMik2r} and the Cauchy-Schwarz inequality, we can get 
\begin{equation}\label{eq_Tik1}
    \mathbb{E}\big(T_{ik,1}-\frac{(w^{\top} \mathbf{M}_{i k} w-\operatorname{tr} \mathbf{M}_{i k})}{(1+\mathbb{E}\operatorname{tr}\mathbf{M}_{ik})^{2}}\big)^{2r}=O(n^{-2r}).
\end{equation}
Combining \eqref{eq_Tik1}, \eqref{eq_Tik2} and \eqref{eq_Tik3}, we can write 
\begin{equation}\label{eq_diffpikk}
    p_{i,kk}-\mathbb{E}p_{i,kk}=:-\frac{(w^{\top} \mathbf{M}_{i k} w-\operatorname{tr} \mathbf{M}_{i k}+\operatorname{tr}\mathbf{M}_{ik}-\mathbb{E}\operatorname{tr}\mathbf{M}_{ik})}{(1+\mathbb{E}\operatorname{tr}\mathbf{M}_{ik})^2}+\mathrm{E}_{ik},
\end{equation}
where the main part  
\begin{equation}\label{eq_defsmik}
    m_{ik}=:-\frac{(w^{\top} \mathbf{M}_{i k} w-\operatorname{tr} \mathbf{M}_{i k}+\operatorname{tr}\mathbf{M}_{ik}-\mathbb{E}\operatorname{tr}\mathbf{M}_{ik})}{(1+\mathbb{E}\operatorname{tr}\mathbf{M}_{ik})^2},
\end{equation}
and the error term 
\begin{equation*}
    \mathrm{E}_{ik}=T_{ik,2}-\big(T_{ik,1}-\frac{(w^{\top} \mathbf{M}_{i k} w-\operatorname{tr} \mathbf{M}_{i k})}{(1+\mathbb{E}\operatorname{tr}\mathbf{M}_{ik})^{2}}\big)-\big(T_{ik,3}-\frac{\operatorname{tr}\mathbf{M}_{ik}-\mathbb{E}\operatorname{tr}\mathbf{M}_{ik}}{(1+\mathbb{E}\operatorname{tr}\mathbf{M}_{ik})^2}\big).
\end{equation*}
It follows that $\mathbb{E}m_{ik}=0$, $\mathbb{E}\mathrm{E}_{ik}=0$ since $\mathbb{E}(p_{i,kk}-\mathbb{E}p_{i,kk})=0$.
Moreover, we have $\mathbb{E}\mathrm{E}_{ik}^{2r}=O(n^{-2r})$ and
\begin{equation*}
    \mathbb{E}m_{ik}^{2r}\lesssim \mathbb{E}(w^{\top} \mathbf{M}_{i k} w-\operatorname{tr} \mathbf{M}_{i k})^{2r}+\mathbb{E}(\operatorname{tr}\mathbf{M}_{ik}-\mathbb{E}\operatorname{tr}\mathbf{M}_{ik})^{2r}\lesssim O(n^{-r})
\end{equation*}
by \eqref{eq_wMikwr} and \eqref{eq_trMik2r}. This further gives 
\begin{equation*}
    \mathbb{E}(p_{i,kk}-\mathbb{E}p_{i,kk})^{2r}\lesssim \mathbb{E}m_{ik}^{2r}+\mathbb{E}\mathrm{E}_{ik}^{2r}\lesssim O(n^{-r})
\end{equation*}
as desired. 

Next we consider \eqref{eq_pikk22r}. Recall $p_{i,kk}=\mathbb{E}p_{i,kk}+m_{ik}+\mathrm{E}_{ik}$ by \eqref{eq_diffpikk}, which further implies
\begin{equation*}
    p_{i,kk}^2=(\mathbb{E}p_{i,kk})^2+(m_{ik}+\mathrm{E}_{ik})^2+2(m_{ik}+\mathrm{E}_{ik})\mathbb{E}p_{i,kk}
\end{equation*}
and 
\begin{equation*}
    \mathbb{E}p_{i,kk}^2=\mathbb{E}(\mathbb{E}p_{i,kk}+m_{ik}+\mathrm{E}_{ik})^2=(\mathbb{E}p_{i,kk})^2+\mathbb{E}m_{ik}^2+O(n^{-2}),
\end{equation*}
where we used the estimates $\mathbb{E}m_{ik}=\mathbb{E}\mathrm{E}_{ik}=0$ and $\mathbb{E}\mathrm{E}_{ik}^2=O(n^{-2})$.
Therefore, we get 
\begin{equation*}
    p_{i,kk}^2-\mathbb{E}p_{i,kk}^2=(m_{ik}+\mathrm{E}_{ik})^2+2(m_{ik}+\mathrm{E}_{ik})\mathbb{E}p_{i,kk}-\mathbb{E}m_{ik}^2+O(n^{-2}),
\end{equation*}
which further gives
\begin{equation*}
    \begin{split}
        \mathbb{E}(p_{i,kk}^2-\mathbb{E}p_{i,kk}^2)^{2r}\lesssim &\mathbb{E}\big((m_{ik}+\mathrm{E}_{ik})^{4r}+(m_{ik}+\mathrm{E}_{ik})^{2r}(\mathbb{E}p_{i,kk})^{2r}+(\mathbb{E}m_{ik}^2)^{2r}+O(n^{-4r})\big)\\
        \lesssim &O(n^{-r}),
    \end{split}
\end{equation*}
completing the proof of \eqref{eq_pikk22r}. It suffices to consider \eqref{eq_pikkpill}. By the decomposition \eqref{eq_diffpikk}, one can write, for $k\ne \ell$, 
\begin{equation*}
    \begin{split}
        p_{i,kk}p_{i,\ell\ell}-\mathbb{E}p_{i,kk}p_{i,\ell\ell}=&(m_{i\ell}+\mathrm{E}_{i\ell})\mathbb{E}p_{i,kk}+(m_{ik}+\mathrm{E}_{ik})\mathbb{E}p_{i,\ell\ell}\\
        &+(m_{ik}+\mathrm{E}_{ik})(m_{i\ell}+\mathrm{E}_{i\ell})-\mathbb{E}(m_{ik}+\mathrm{E}_{ik})(m_{i\ell}+\mathrm{E}_{i\ell}).
    \end{split}
\end{equation*}
It follows that
\begin{equation*}
    \begin{split}
        \mathbb{E}(p_{i,kk}p_{i,\ell\ell}-\mathbb{E}p_{i,kk}p_{i,\ell\ell})^{2r}\lesssim \mathbb{E}(m_{ik}+\mathrm{E}_{ik})^{2r}\lesssim O(n^{-r}),
    \end{split}
\end{equation*}
completing the proof of this proposition.         \qed

\begin{remark}\label{remark_lowerbound}
    We can show that $\mathbb{E}m_{ik}^2\ge c/n$ for some constant $c>0$ if $i$ is large comparing with $n$. To this end, consider the simplest case $\mathbf{A}=\mathbf{I}$ and notice 
    \begin{equation*}
        \mathbb{E}(w^{\top} \mathbf{M}_{i k} w-\operatorname{tr} \mathbf{M}_{i k})^2\ge \mathbb{E}\operatorname{tr}(\mathbf{M}_{ik}^2)=\mathbb{E}\operatorname{tr}(\mathbf{W}_{i,n-1}^{-2})=\frac{i(n-2)}{(n-i-1)(n-i-2)(n-i-4)},
    \end{equation*}
    where we used the fact that $\mathbb{E}\operatorname{tr}(\mathbf{W}_{i,n}^{-2})=\frac{i(n-1)}{(n-i)(n-i-1)(n-i-3)}$ for $\mathbf{W}_{i,n}\sim \text{Wishart}~(\mathbf{I}_i,n)$. On the other hand, for any $1\le i\le p-1$, we have 
    \begin{equation*}
        \mathbb{E}(\operatorname{tr}\mathbf{M}_{ik}-\mathbb{E}\operatorname{tr}\mathbf{M}_{ik})^2\le \frac{4}{n}.
    \end{equation*}
Thus, by the elementary inequality $(a+b)^2\ge a^2/2-b^2$, one concludes that
\begin{equation*}
    \begin{split}
        \mathbb{E}m_{ik}^2\ge &c\big(\frac{1}{2}\mathbb{E}(w^{\top} \mathbf{M}_{i k} w-\operatorname{tr} \mathbf{M}_{i k})^2-\mathbb{E}(\operatorname{tr}\mathbf{M}_{ik}-\mathbb{E}\operatorname{tr}\mathbf{M}_{ik})^2)\\
        \ge & c\big(\frac{i(n-2)}{2(n-i)^3}-\frac{4}{n}\big)\ge c/n,
    \end{split}
\end{equation*}
if $i\ge 2n/3$ for instance. This also suggests that $\mathbb{E}(p_{i,kk}-\mathbb{E}p_{i,kk})^2\asymp 1/n$ even for the simplest case $\mathbf{A}=\mathbf{I}$.
\end{remark}

\bibliographystyle{elsarticle-harv}
\bibliography{reference}

@article{jiang2025asymptoticaos,
  title={Asymptotic distributions of largest {P}earson correlation coefficients under dependent structures},
  author={Jiang, Tiefeng and Pham, Tuan},
  journal={The Annals of Statistics},
  volume={53},
  number={3},
  pages={907--928},
  year={2025},
  publisher={Institute of Mathematical Statistics}
}

@article{jiang2025asymptoticaossupp,
  title={Supplement to “{A}symptotic distributions of largest {P}earson correlation coefficients under dependent structures”},
  author={Jiang, Tiefeng and Pham, Tuan},
  year={2025},
  publisher={Institute of Mathematical Statistics}
}

@article{heiny2026maximum,
  title={Maximum of sparsely equicorrelated {G}aussian fields and applications},
  author={Heiny, Johannes and Jiang, Tiefeng and Pham, Tuan and Qi, Yongcheng},
  journal={arXiv:2603.05306},
  year={2026}
}

@article{gusakova2023volume,
  title={The volume of random simplices from elliptical distributions in high dimension},
  author={Gusakova, Anna and Heiny, Johannes and Th{\"a}le, Christoph},
  journal={Stochastic Processes and their Applications},
  volume={164},
  pages={357--382},
  year={2023},
  publisher={Elsevier}
}

@book{holgersson2020recent,
  title={Recent developments in multivariate and random matrix analysis},
  author={Holgersson, Thomas and Singull, Martin},
  year={2020},
  publisher={Springer},
address={Cham, Switzerland}
}

@article{pielaszkiewicz2020mixtures,
  title={Mixtures of traces of Wishart and inverse Wishart matrices},
  author={Pielaszkiewicz, Jolanta and Holgersson, Thomas},
  journal={Communications in Statistics-Theory and Methods},
  volume={50},
  number={21},
  pages={5084--5100},
  year={2020},
  publisher={Taylor \& Francis}
}

@article{hu2019high,
  title={High-dimensional covariance matrices in elliptical distributions with application to spherical test},
  author={Hu, Jiang and Li, Weiming and Liu, Zhi and Zhou, Wang},
  journal={The Annals of Statistics},
  volume={47},
  number={1},
  pages={527--555},
  year={2019},
  publisher={JSTOR}
}

@article{onatski2013asymptotic,
  title={Asymptotic power of sphericity tests
for high-dimensional data},
  author={Onatski, Alexei and Moreira, Marcelo J and Hallin, Marc},
  journal={The Annals of Statistics},
  volume={41},
  number={3},
  pages={1204--1231},
  year={2013}
}

@article{girko1998refinement,
  title={A refinement of the central limit theorem for random determinants},
  author={Girko, Vyacheslav L},
  journal={Theory of Probability and Its Applications},
  volume={42},
  number={1},
  pages={121--129},
  year={1998},
  publisher={SIAM}
}

@book{bai2010spectral,
  title={Spectral analysis of large dimensional random matrices},
  author={Bai, Zhidong and Silverstein, Jack W},
  address={New York},
  year={2010},
  publisher={Springer}
}

@article{adamczak2024limit,
  title={Limit theorems for the volumes of small codimensional random sections of $\ell_p^n$-balls},
  author={Adamczak, Rados{\l}aw and Pivovarov, Peter and Simanjuntak, Paul},
  journal={The Annals of Probability},
  volume={52},
  number={1},
  pages={93--126},
  year={2024},
  publisher={Institute of Mathematical Statistics}
}

@article{barany2007central,
  title={Central Limit Theorems for Gaussian Polytopes},
  author={B{\'a}r{\'a}ny, Imre and Vu, Van},
  journal={The Annals of Probability},
  pages={1593--1621},
  year={2007},
  publisher={JSTOR}
}

@article{johnstone2001distribution,
  title={On the distribution of the largest eigenvalue in principal components analysis},
  author={Johnstone, Iain M},
  journal={The Annals of Statistics},
  volume={29},
  number={2},
  pages={295--327},
  year={2001},
  publisher={Institute of Mathematical Statistics}
}

\end{document}